\documentclass[letterpaper, 10 pt, conference]{ieeeconf}  

\IEEEoverridecommandlockouts                              

\usepackage{epsfig} 
\usepackage{mathptmx} 
\usepackage{times} 
\usepackage{amsmath} 
\usepackage{amssymb}  
\usepackage{graphicx}      
\usepackage{cite}
\usepackage{amsfonts}

\usepackage{tikz}
\usepackage{pgfplots,pgfkeys}
\usetikzlibrary{shapes.geometric, arrows, arrows.meta}

\usepackage{comment}

\usepackage{centernot}

\newcommand{\tr}{\operatorname{tr\,}}
\newcommand{\id}{\operatorname{id\,}}

\newtheorem{remark}{Remark}
\newtheorem{theorem}{Theorem}
\newtheorem{definition}{Definition}
\newtheorem{corollary}{Corollary}

\title{\LARGE \bf
Integral gains for non-autonomous Wazewski systems
}

\author{Ivan Atamas, Sergey Dashkovskiy and Vitalii Slynko$^{1}$
\thanks{$^{1}$Authors are with the Institute of Mathematics,
        University of Würzburg, Emil-Fischer-Str. 40, 97074 Würzburg, Germany, Email: name.sutname@uni-wuerzburg.de}
}

\begin{document}

\maketitle
\thispagestyle{empty}
\pagestyle{empty}

\begin{abstract}

In this work we consider linear non-autonomous systems of Wazewski type on Hilbert spaces and provide a
new approach to study their stability properties by means of a 
decomposition into subsystems and conditions implied on the interconnection 
properties. These conditions are of the small-gain type but the appoach is based on a conceptually new notion which we call integral gain. This notion is introduced for the first time in this paper. We compare our approach with known results from the literature and demonstrate advantages of our results.

\end{abstract}

\section{Introduction}


Non-autonomous systems appear in modeling of nonstationary processes or, for example, in case of time dependent stabilizing controllers, see e.g.,  \cite{LLKL18},\cite{SlB21}, \cite{STB21}, \cite{SuD22}, \cite{ZaD24}. In modern applications related to networks such systems can be of large-scale and have a time varying interaction structure, which may lead to a destabilization even if all disconnected are stable.

A standard approach to study stability properties of a large system is a
 decomposition \cite{LiS10} to several (say $r\in\mathbb N$) smaller subsystems, for which these properties are easier to verify. For example, one assumes that Lyapunov-type functions $V_1,\dots, V_r$ are available for the disconnected subsystems and tries to construct a Lyapunov function for the whole system as their linear  \cite{LiS10} or nonlinear \cite{DRW10} combination. This construction is possible under suitable restrictions of the interaction intensity between subsystems such as small-gain conditions. For two interconnected non-autonomous systems such results are available in \cite{KaT04}.
 
 Another method \cite{Bel62},\cite{Mat62} is to consider the vector Lyapunov function $V=(V_1,\dots,V_r)^T$ and to construct a differential inequality of the form $\dot V\le g(t,V)$  with a quasi-monotone vector-function $g$. The latter property, also called Wazewski property, gives the possibility to pass to a comparison system, stability of which guarantees stability of the original interconnection. In many cases this comparison system is of Wazewski type (that is with quasi-monotone $g$), for which monotonicity of solutions with respect to initial states holds \cite{RHL77}. This implies the positivity property and makes the stability investigation of the comparison system essentially easier, than of the original one. Let us note that such properties as positivity and monotonicity appear in many dynamical systems as
a consequence of physical, biological or economic laws. Hence stability investigation of Wazevski systems is of interest for itself and not only as an intermediate step. In this work we consider linear non-autonomous Wazewski systems, for which almost no results exist.

A problem which we face in case of non-autonomous systems is that interaction between subsystems can be strong (with high gains) on certain time intervals, which can make the application of the known small-gain conditions impossible, however if these intervals of intensive interactions are short or even confined to a finite time, then it can be expected that the whole system remains stable. 

To overcome this difficulty we introduce the notion of integral gains in this work for the first time. These gains coincide with usual gains in case of autonomous systems. Moreover we use a completely different approach (see \eqref{SLF}) for a construction of a Lyapunov function for the interconnection. 
Based on these ideas we derive a new stability condition of the small-gain type and demonstrate its advantages over the existing results by means of examples. In particular, our approach can be applied in certain situations, where ather results are not applicable.

The paper is organized as follows. All necessary definitions and notation as well as the problem statement are introduced in in Section \ref{preliminaries}. The main result and a sketch of its proof are provided in Section \ref{main-result}. Illustrative examples are provided in Sections \ref{ex1} and \ref{ex2}. Concluding remarks are collected in Section \ref{conclusions}.

\section{Preliminaries}\label{preliminaries}
In this work we restrict our investigations to the case of two interacting subsystems. To this end let $\mathcal H={\mathcal H_1}\bigoplus{\mathcal H_2}$ be the direct product of two real Hilbert spaces. For some $\tau\in\mathbb R$ let $A\,\,:\,[\tau,\infty)\to L(\mathcal H)$ be strongly measurable and Bochner integrable on any finite interval $[a,b]\subset[\tau,\infty)$ function.
We consider the following non-autonomous differential equation
\begin{equation}\label{1}
\gathered
\dot x=A(t)x,\quad x\in\mathcal H,\quad t_0\in[\tau,\infty)
\endgathered
\end{equation}
with initial condition $x(t_0)=x_0$, $x_0\in\mathcal H$. We say that $x$ is a solution to this Cauchy problem if $x\,:[t_0,\infty)\to\mathcal H$ is continuous and satisfies the integral equation
\begin{equation}\label{1*}
x(t)=x_0+\int\limits_{t_0}^tA(s)x(s)\,ds,
\end{equation}
where the integral is in the sense of Bochner.
To stress the initial data we also write $x(t)=x(t;t_0,x_0)$. This solution can be represented as $x(t)=x(t;t_0,x_0)=\Omega(t;t_0)x_0$, by means of the evolution operator $\Omega(t,t_0)$ of the system \eqref{1}, which satisfy the following properties, see e.g., \cite[Chapter 3.1]{DaK74}:  $\Omega\in C([\tau,\infty)\times[\tau,\infty);L(\mathcal H))$ and
\\
1) $\Omega(t,t)=\id$, $\Omega(t,t_0)=\Omega(t,s)\Omega(s,t_0)$ for all $t$, $s$, $t_0\in[\tau,\infty)$;
\\
2) and for almost all $t,s\in[\tau,\infty)$ it holds that
\begin{equation}\label{1**}
\gathered
\frac{\partial}{\partial t}\Omega(t,s)=A(t)\Omega(t,s),\\ 
\frac{\partial}{\partial s}\Omega(t,s)=-A(s)\Omega(t,s).
\endgathered
\end{equation}
A non-empty, closed, convex subset $K\subset\mathcal H$ is called a cone, if $\mathbb R_+K\subset K$ and $K\bigcap(-K)=\{0\}$ hold. 
A cone $K$ is called generating if any $x\in\mathcal H$ can be written as $x=x^+-x^-$ for some
$x^{\pm}\in K$. In a Hilbert space any generating cone is nonflat, that is there exists a constant $a_K>0$, independent on $x$, such that in the representation $x=x^+-x^-$ the vectors $x^{\pm}$ can be chosen so that 
\begin{equation}\label{nonflat}
\|x^{\pm}\|\le a_K\|x\|.    
\end{equation}
The set $K^*=\{f\in L(\mathcal H;\mathbb R)\,:\,f(K)\subset\mathbb R_+\}\subset\mathcal H^*=\mathcal H$
is called adjoint cone. If $K$ is generating, then $K^*$ is normal, i.e., there is a constant $b_K>0$ such that $x\in K$, $y-x\in K$ imply $\|x\|\le b_K\|y\|$. A cone $K$ is called selfadjoint if $K^*=K$. In a Hilbert space any selfadjoint cone is normal \cite{KLS89}.
\begin{definition}
Let $K$ be a generating and selfadjoint cone in $\mathcal H$.
    System \eqref{1} is called Wazewski \cite{Mar11} with respect to $K$ (also called monotone \cite{LaP94}) if its evolution operator $\Omega$ satisfies $\Omega(t,s)K\subset K$ for all $t\ge s\ge\tau$.
\end{definition}
\begin{definition}
We say that \eqref{1} is stable in $K$ if for any $\varepsilon>0$ and any $t_0\ge\tau$ there is a $\delta=\delta(\varepsilon,t_0)>0$ such that
$x_0\in B_{\delta}\cap K\implies \|x(t;t_0,x_0)\|<\varepsilon,$ $t\ge t_0$;
if, in addition, the $\delta$ can be chosen independent on $t_0$, then \eqref{1} is called uniformly stable in $K$; 

asymptotically stable in $K$ if it is stable in $K$ and there is $\eta=\eta(t_0)>0$ such that $x_0\in B_{\eta}\cap K\implies \lim_{t\to\infty}\|x(t;t_0,x_0)\|=0$; if, in addition, $\eta$ can be chosen independent on $t_0$, then \eqref{1} is called uniformly asymptotically stable in $K$.

If any of the above properties holds with $\mathcal H$ on the place of $K$, then we drop ''in $K$'' in their definitions.
\end{definition}
Let system \eqref{1} be decomposed in two subsystems as follows
\begin{equation}\label{2}
\dot x_i=A_{ii}(t)x_i+A_{ij}(t)x_j,\quad i\ne j, \quad i,j=1,2,
\end{equation}
with $x_i\in\mathcal H_i$, $A_{ij}\in L(H_j,H_i)$ for $i,j=1,2$.
Let $K_i$ be a solid selfadjoint cone in $\mathcal H_i$, then $K=K_1\bigoplus K_2$ is a solid and selfadjoint cone in $\mathcal H$.

Let $\Omega_i\in C([\tau,\infty)\times[\tau,\infty);L(\mathcal H_i))$ denote the evolution operator of the decoupled subsystem
\begin{equation}\label{3}
\gathered
\dot y_i=A_{ii}(t)y_i,\quad y_i\in\mathcal H_i,\quad i=1,2
\endgathered
\end{equation}

It can be verified by definition that \eqref{1}, written as \eqref{2}, is a Wazewski system for the cone $K$,
if and only if 
\begin{equation}\label{wazewski-conditions}
\Omega_i(t,s)K_i\subset K_i \quad \text{ and }\quad A_{ij}(t)K_j\subset K_i,\quad t\ge s\ge \tau.    
\end{equation}
\begin{definition}\label{positiveness}
An operator $O\in L(\mathcal H_j;\mathcal H_i)$ is called positive, if it satisfies $OK_j\subset K_i$ and it is denoted by $O\ge 0$.
\end{definition}
For the evolution operator $\Omega_i$ of \eqref{3} let $\alpha_i$, $\beta_i$, $\gamma_i$, $\delta_i\in C([\tau,\infty);\mathbb R_{>0})$ be such that 
\begin{equation}\label{4}
\gathered
\|\Omega_i(t,s)\|\le \alpha_i(t)\beta_i(s),\quad t\ge s\ge \tau,\\
\|\Omega_i(t,s)\|\le (\gamma_i(t))^{-1}(\delta_i(s))^{-1},\quad s\ge t\ge \tau,
\endgathered
\end{equation}
for example we can fix any $p\in(s,t)$ and take  $\alpha_i(t)=\|\Omega_i(t,p)\|$, $\beta_i(s)=\|\Omega_i(p,s)\|$ due to $\Omega_i(t,s)=\Omega_i(t,p)\Omega_i(p,s)$.

Let $q_i\in C([\tau,\infty);\mathbb R_{>0})$ be suitable weight functions guaranteeing convergence of the next integrals for $t\in[\tau,\infty)$
\begin{equation}\label{weights_q_i}
\gathered
\phi_i(t):=\gamma_i^2(t)\int\limits_t^{\infty}q_i(s)\delta_i^2(s)\,ds, \\
g_i(t):=\beta_i^2(t)\int\limits_t^{\infty}q_i(s)\alpha_i^2(s)\,ds.
\endgathered
\end{equation}
\section{Main Results}\label{main-result}
For simplicity and clearness in this work we consider the case of $r=2$ interconnected subsystems in \eqref{2}. 
An extension for $r\in\mathbb N$ will developed elsewhere.
We introduce the following notation for $1\le i\neq j\le2$:
linear weighted integral gains for the interconnection \eqref{2} are defined as
\begin{equation}\label{gains-ii}
\gathered
\pi_{ii}(t_0)=2\sup\limits_{t\ge t_0}\omega_i(t)\beta_i^2(t)\\
\times\int\limits_{t}^{\infty}\alpha_i(s)\|A_{ji}(s)\|\beta_j(s)
\int\limits_s^{\infty}\frac{\alpha_i(p)\alpha_j(p)\|A_{ij}(p)\|}{\omega_i(p)}\,dp\,ds,
\endgathered
\end{equation}
\begin{equation}\label{gains-ij}
\gathered
\pi_{ji}(t_0)=2\sup\limits_{t\ge t_0}\omega_i(t)\beta_i^2(t)\\
\times\int\limits_{t}^{\infty}\alpha_i(s)\|A_{ji}(s)\|\beta_j(s)
\int\limits_s^{\infty}\frac{\alpha_i(p)\alpha_j(p)\|A_{ji}(p)\|}{\omega_j(p)}\,dp\,ds,
\endgathered
\end{equation}
where $\omega_i$ are suitable weight functions, which can help to enable convergence of the integrals. Also for $f\in C([t_0,T];L(\mathcal H_i))$ we introduce the weighted norm 
\begin{equation}
    \|f\|_{\omega_i,T}:=\max\limits_{t\in [t_0,T]}\omega_i(t)\|f(t)\|.
\end{equation}
\begin{theorem}\label{th:main-result}
Let \eqref{1} be a Wazewski system with respect to a selfadjoint solid cone $K$ and written as \eqref{2} with $r=2$. Let $\alpha_i, \beta_i, \gamma_i, \delta_i$ be as in \eqref{4} and $\phi_i, g_i, q_i, \omega_i, \pi_{ij}$ as in \eqref{weights_q_i},\eqref{gains-ii},\eqref{gains-ij}.
If the spectral radius of the matrix $\Pi(t_0)\in\mathbb R^{2\times2}$ defined by the weighted integral gains $\pi_{ij}(t_0)$ satisfies 
\begin{equation}\label{ISGC}
r_\sigma(\Pi(t_0))<1
\end{equation}
then solutions to \eqref{1} satisfy the estimate
\begin{equation}\label{estimate-for-solutions}
\|x(t;t_0,x_0)\|\le  2a_K\sqrt{\frac{h(t_0,t_0)}{\phi(t)}}\exp\Big(-\int\limits_{t_0}^t\frac{q(s)}{2h(s,t_0)}\,ds\Big)\|x_0\|,
\end{equation}
where $a_K$ is from \eqref{nonflat}, $q(t):=\min\{q_1(t),q_2(t)\}$, $\phi(t):=\min\{\phi_1(t),\phi_2(t)\}$ and
$h(t,t_0):=\max\{h_{11}(t,t_0)+h_{12}(t,t_0),h_{22}(t,t_0)+h_{12}(t,t_0)\}$ with
\begin{equation*}
\gathered
h_{ii}(t,t_0):=\frac{1}{1-\tr\Pi(t_0)+\det\Pi(t_0)}\\
\times\Big((1-\pi_{ii}(t_0))\frac{\|g_i\|_{\omega_i,t}}{\omega_i(t)}+
\pi_{ij}(t_0)\frac{\|g_j\|_{\omega_j,t}}{\omega_i(t)}\Big),\\
h_{12}(t,t_0):=
\beta_1(t)\beta_2(t)\int\limits_{t}^{\infty}\alpha_1(s)\alpha_2(s)\Big(\|A_{12}(s)\|h_{11}(s,t_0)\\
+\|A_{21}(s)\|h_{22}(s,t_0)\,\Big)\,ds,
\endgathered
\end{equation*}
and all $h_{ij}$ are assumed to be well defined for any $t\ge t_0$. 

Moreover, if $\inf_{t\ge t_0}\phi(t)>0$ and $h(t_0,t_0)<\infty$, then \eqref{1} is stable. If, in addition, 
either $\phi(t)\to \infty$ for $t\to+\infty$ or
\begin{equation}\label{stabC}
\int\limits_{t_0}^t\frac{q(s)}{h(s,t_0)}\,ds\to+\infty,\quad t\to\infty,
\end{equation}
then \eqref{1} is asymptotically stable.
If, in addition, $\sup_{t_0\ge \tau}h(t_0,t_0)<+\infty$ holds,
then \eqref{1} is uniformly asymptotically stable.
\end{theorem}
\begin{remark}
Condition \eqref{ISGC} is the small-gain condition for integral gains.
    Note that $r_\sigma(\Pi(t_0))<1\;\Leftrightarrow\; -1+\tr\Pi(t_0)<\det\Pi(t_0)<1$.
\end{remark}
\begin{proof}[Sketch]
In the first step we will prove estimates assuring desired stability properties with respect to $K$. In the last step this will be extended to the whole state space.

{\it Step 1. Matrix-valued Lyapunov function.} \\
We look for a Lyapunov function for system \eqref{1} in the form
\begin{equation}\label{SLF}
v(t,x_1,x_2)=v_{11}(t,x_1)+2v_{12}(t,x_1,x_2)+v_{22}(t,x_2),
\end{equation}
with $v_{ii}(t,x_i)=\langle x_i,P_{ii}(t)x_i\rangle_{\mathcal H_i}$,
$v_{12}(t,x_1,x_2)=v_{21}(t,x_1,x_2)=\langle x_1,P_{12}(t)x_2\rangle_{\mathcal H_1}$, and $P_{ij}(t):\mathcal H_j\to\mathcal H_i$ to be defined later. 
Its derivative along solutions of \eqref{1} is 
\begin{equation*}
\gathered
\dot v(t,x_1,x_2)= \langle x_1,(\dot P_{11}(t)+A_{11}^{*}(t)P_{11}(t)+P_{11}(t)A_{11}(t)\\
+(P_{12}(t)A_{21}(t)+A_{21}^{*}(t)P_{12}^{T}(t)))x_1\rangle_{\mathcal H_1}
\\
+2\langle x_1,((\dot P_{12}(t)+A_{11}^{*}(t)P_{12}(t)+P_{12}(t)A_{22}(t))\\
+P_{11}(t)A_{12}(t)+A_{21}^{*}(t)P_{22}(t))x_2\rangle_{\mathcal H_1}\\
+\langle x_2,((\dot P_{22}(t)+A_{22}^{*}(t)P_{22}(t)+P_{22}(t)A_{22}(t))\\
+(A_{12}^{*}(t)P_{12}(t)+P_{12}^{*}(t)A_{12}(t)))x_2\rangle_{\mathcal H_2}
\endgathered
\end{equation*}
which is negative definite, if $P_{ij}(t)$, $i,j=1,2$ satisfy the next differential equations (we omit the argument $t$ for brevity)
\begin{equation}\label{DUMFL}
\gathered
\dot P_{11}+A_{11}^{*}P_{11}+P_{11}A_{11}+(P_{12}A_{21}+A_{21}^{*}P_{12}^{*})=-\id q_1,
\\
\dot P_{12}+A_{11}^{*}P_{12}+P_{12}A_{22}+P_{11}A_{12}+A_{21}^{*}P_{22}\quad=\quad 0,\\
\dot P_{22}+A_{22}^{*}P_{22}+P_{22}A_{22}+(A_{12}^{*}P_{12}+P_{12}^{*}A_{12})=-\id q_2,
\endgathered
\end{equation}
where the functions $q_i$ are from \eqref{weights_q_i}, so that
\begin{equation}\label{DLF}
\gathered
\dot v(t,x_1,x_2)=-q_1(t)\|x_1\|^2-q_2(t)\|x_2\|^2\le-q(t)\|x\|^2.
\endgathered
\end{equation}
Equations \eqref{DUMFL} can be converted to integral equations
{\small
\begin{equation}\label{7bis}
\gathered
P_{11}(t)=\int\limits_{t}^{\infty}\Omega_1^{*}(s,t)
(\id q_1(s)+P_{12}(s)A_{21}(s)+A_{21}^{*}(s)P_{12}^{*}(s))\Omega_1(s,t)\,ds,\\
P_{22}(t)=\int\limits_{t}^{\infty}\Omega_2^{*}(s,t)
(\id q_2(s)+A_{12}^{*}(s)P_{12}(s)+P_{12}^{*}(s)A_{12}(s))\Omega_2(s,t)\,ds,\\
P_{12}(t)=\int\limits_{t}^{\infty}\Omega_1^{*}(s,t)
(P_{11}(s)A_{12}(s)+A_{21}^{*}(s)P_{22}(s))\Omega_2(s,t)\,ds.
\endgathered
\end{equation}
}
By direct calculations it is seen that any solution to \eqref{7bis}
satisfies \eqref{DUMFL}. Note that the last equation can be excluded by its substitution to the other two. Now need to derive conditions guaranteeing the existence of a solution to the remaining two equations.

{\it Step 2. System of integral equations for diagonal elements.}\\
The remaining two equations for $P_{11}$, $P_{22}$ can be written as
\begin{equation}\label{A1-A2}
P_{ii}(t)=Q_i(t)+\mathfrak{F}_i(P_{11},P_{22})(t),\quad i=1,2,\quad t\in[t_0,\infty)
\end{equation}
with some suitable $Q_i(t)\ge0$ and $\mathfrak{F}_i(P_{11},P_{22})(t)\ge0$, which can be explicetly written, but are omitted for brevity.

The existence of a solution to the system \eqref{A1-A2}
can be proved by the generalized Banach theorem \cite[Sections 11.3 and 12.1]{Col66} for pseudometric spaces, see \cite[Sections 3.3]{Col66}. 
The application of this fixed point theorem is possible due to the condition \eqref{7bis}.

Note that matrix $\Pi(t_0)$ is positive, since all its components are non-negative
\begin{equation*}
\gathered
\Pi(t_0):=
\begin{pmatrix}
\pi_{11}(t_0)&\pi_{12}(t_0)\\
\pi_{12}(t_0)&\pi_{22}(t_0)
\end{pmatrix}.
\endgathered
\end{equation*}
The  fixed-point theorem from \cite[Sections 11.3 and 12.1]{Col66} implies due to $r_{\sigma}(\Pi(t_0))<1$ the existence and uniqueness of a solution $(P_{11},P_{22})$ to \eqref{A1-A2}, which satisfies (note that $(I-\Pi(t_0))^{-1}$ is also positive)
\begin{equation*}
\gathered
\begin{pmatrix}
\|P_{11}\|_{\omega_1,T}\\
\|P_{22}\|_{\omega_2,T}
\end{pmatrix}\le
(I-\Pi(t_0))^{-1}
\begin{pmatrix}
\|Q_1\|_{\omega_1,T}\\
\|Q_2\|_{\omega_2,T}
\end{pmatrix}\\
\qquad\qquad\qquad\stackrel{\eqref{weights_q_i}}{\le}
(I-\Pi(t_0))^{-1}
\begin{pmatrix}
\|g_1\|_{\omega_1,T}\\
\|g_2\|_{\omega_2,T}
\end{pmatrix}
\endgathered
\end{equation*}
\begin{equation*}
\gathered
=\frac{1}{1-\tr\Pi(t_0)+\det\Pi(t_0)}\\
\times
\begin{pmatrix}
1-\pi_{22}(t_0)&
\pi_{12}(t_0)
\\
\pi_{21}(t_0)&
1-\pi_{11}(t_0)
\end{pmatrix}
\begin{pmatrix}
\|g_1\|_{\omega_1,T}\\
\|g_2\|_{\omega_2,T}
\end{pmatrix},
\endgathered
\end{equation*}
or componentwise
\begin{equation*}
\|P_{11}\|_{\omega_1,T}\le\frac{(1-\pi_{11}(t_0))\|g_1\|_{\omega_1,T}+
\pi_{12}(t_0)\|g_2\|_{\omega_2,T}}{1-\tr\Pi(t_0)+\det\Pi(t_0)},
\end{equation*}
\begin{equation*}
\|P_{22}\|_{\omega_2,T}\le\frac{\pi_{21}(t_0)\|g_1\|_{\omega_1,T}+(1-\pi_{22}(t_0))\|g_2\|_{\omega_1,T}}{1-\tr\Pi(t_0)+\det\Pi(t_0)},
\end{equation*}
Hence for any $t\ge t_0$ the next inequalities hold
\begin{equation}\label{h11}
\|P_{11}(t)\|\le\frac{(1-\pi_{11}(t_0))\frac{\|g_1\|_{\omega_1,t}}{\omega_1(t)}+
\pi_{12}(t_0)\frac{\|g_2\|_{\omega_2,t}}{\omega_1(t)}}{1-\tr\Pi(t_0)+\det\Pi(t_0)}
=:h_{11}(t,t_0),
\end{equation}
\begin{equation}\label{h22}
\|P_{22}(t)\|\le\frac{\pi_{21}(t_0)\frac{\|g_1\|_{\omega_1,t}}{\omega_2(t)}+(1-\pi_{22}(t_0))\frac{\|g_2\|_{\omega_2,t}}{\omega_2(t)}}{1-\tr\Pi(t_0)+\det\Pi(t_0)}
=:h_{22}(t,t_0).
\end{equation}
Now from the last equation in \eqref{7bis} we obtain
\begin{equation}\label{h12}
\gathered
\|P_{12}(t)\|\le\int\limits_{t}^{\infty}\|\Omega_1^{*}(s,t)\|
(\|P_{11}(s)\|\|A_{12}(s)\|\\
+\|A_{21}^{*}(s)\|\|P_{22}(s)\|)\|\Omega_2(s,t)\|\,ds\\
\le\beta_1(t)\beta_2(t)\int\limits_{t}^{\infty}\alpha_1(s)\alpha_2(s)(\|A_{12}(s)\|h_{11}(s,t_0)
\\+\|A_{21}(s)\|h_{22}(s,t_0))\,ds=:h_{12}(t,t_0).
\endgathered
\end{equation}

Recall that solutions to integral equations \eqref{A1-A2} can be approximated iteratively. If we choose 
$$P_{ii}^{(0)}(t)\ge 0 \quad\text{ for }\quad t\ge t_0$$ as the initial iteration, then due to  $$\Omega_i(t,s)\ge 0,\qquad A_{ij}(t)\ge 0,\qquad i\ne j$$ it follows for any subsequent iteration that $P_{ii}^{(k)}(t)\ge 0$ for $k\in\mathbb Z_+$, $t\ge t_0$.
This implies that solutions to \eqref{A1-A2} satisfy
$P_{ii}(t)\ge 0$ for $t\ge t_0$, $i=1,2$. Hence, also $P_{12}(t)\ge 0$ holds for  $t\ge t_0$.

Moreover, from \eqref{A1-A2} the next estimate follows
\begin{equation}\label{lower-bound-for-Pii}
P_{ii}(t)\ge \int\limits_{t}^{\infty}q_i(s)\Omega_i^{*}(s,t)\Omega_i(s,t)\,ds,\quad t\ge t_0,\quad i=1,2.
\end{equation}
These observations are useful for the derivation of the upper and lower estimates for  $v(t,x_1,x_2)$.

{\it Step 3. Estimation of $v(t,x_1,x_2)$.} 

From \eqref{SLF} and \eqref{h11}-\eqref{h12} it follows that
\begin{equation}\label{Ineq1}
\gathered
v(t,x_1,x_2)\\
\le \|P_{11}(t)\|\|x_1\|^2+2\|P_{12}(t)\|\|x_1\|\|x_2\|+\|P_{22}(t)\|\|x_2\|^2\\
\le  h_{11}(t,t_0)\|x_1\|^2+
2h_{12}(t,t_0)\|x_1\|\|x_2\|
+h_{22}(t,t_0)\|x_2\|^2\\
\le (h_{11}(t,t_0)+h_{12}(t,t_0))\|x_1\|^2+
(h_{22}(t,t_0)+h_{12}(t,t_0))\|x_2\|^2\\
\le h(t,t_0)(\|x_1\|^2+\|x_2\|^2).
\endgathered
\end{equation}
To estimate the Lyapunov function $v(t,x_1,x_2)$ from below, we note that for $x_i\in K_i$ the following estimates hold due to \eqref{lower-bound-for-Pii}
\begin{equation*}
\gathered
\langle x_i,P_{ii}(t)x_i\rangle_{\mathcal H_i}\ge \int\limits_{t}^{\infty}q_i(s)\|\Omega_i(s,t)x_i\|^2\,ds\\
\ge
\delta_i^2(t)\int\limits_{t}^{\infty}q_i(s)\gamma_i^2(s)\,ds\|x_i\|^2\ge\phi_i(t)\|x_i\|^2,\\
\langle x_1,P_{12}(t)x_2\rangle_{\mathcal H_1}\ge 0
\endgathered
\end{equation*}
for all $t\ge t_0$, from which we obtain a coercivity type estimate
\begin{equation}\label{Ineq2}
\gathered
v(t,x_1,x_2)\ge \phi_1(t)\|x_1\|^2+\phi_2(t)\|x_2\|^2\\
\qquad\quad\ge\phi(t)(\|x_1\|^2+\|x_2\|^2).\\
\endgathered
\end{equation}

{\it Step 4. Estimation of solutions in $K$} \\
The estimations \eqref{Ineq1} and \eqref{Ineq2} of the Lyapunov function and \eqref{DLF} for its derivative allow to estimate solutions, from wich suffitient stability conditions can be derived easily.
Namely, from \eqref{DLF} and \eqref{Ineq1} we obtain
\begin{equation*}
\gathered
\dot v(t,x_1,x_2)=-q(t)(\|x_1\|^2+\|x_2\|^2)\le-\frac{q(t)}{h(t,t_0)}v(t,x_1,x_2).
\endgathered
\end{equation*}
Which implies the estimate
\begin{equation}\label{Ineq3}
\gathered
v(t,x_1(t),x_2(t))\le v(t_0,x_1(t_0),x_2(t_0))\exp\Big(-\int\limits_{t_0}^t\frac{q(s)}{h(s,t_0)}\,ds\Big).
\endgathered
\end{equation}
Now from \eqref{Ineq1} and \eqref{Ineq2} we derive the desired estimate for solutions with initial state $x_0\in K$
\begin{equation*}\label{Ineq3}
\|x(t)\|\le \sqrt{\frac{h(t_0,t_0)}{\phi(t)}}\exp\Big(-\int\limits_{t_0}^t\frac{q(s)}{2h(s,t_0)}\,ds\Big)\|x_0\|.
\end{equation*}

{\it Step 5. Estimation of all solutions.}

For arbitrary $x_0\in\mathcal H$ there exist $x^{\pm}_0\in K$ such that
$x_0=x_0^+-x_0^-$ and  $\|x^{\pm}_0\|\le a_K\|x_0\|$, $a_K>0$. By the linearity of the system using \eqref{Ineq3} we obtain the desired estimate
\begin{equation*}
\gathered
\|x(t;t_0,x_0)\|\le \|x(t;t_0,x_0^+)-x(t;t_0,x_0^-)\|\\
\le
\|x(t;t_0,x_0^+)\|+\|x(t;t_0,x_0^-)\|\\
\le 2a_K\sqrt{\frac{h(t_0,t_0)}{\phi(t)}}\exp\Big(-\int\limits_{t_0}^t\frac{q(s)}{2h(s,t_0)}\,ds\Big)\|x_0\|
\endgathered
\end{equation*}
{\it Step 6. Stability properties.}
Basic analysis observations imply the desired stability properties under the corresponding additional conditions given at the end of the theorem. We omit the details for brevity.
\end{proof}
As a consequence we can derive much easier conditions by taking the limit for $t_0\to\infty$ in $\Pi(t_0)$. To this end we estimate the functions $h,h_{ii},h_{12}$ from Theorem \ref{th:main-result} as follows.
\begin{corollary}\label{sledstvie}
Let there exist $\Pi:=\lim\limits_{t_0\to\infty}\Pi(t_0)$ such that $r_\sigma(\Pi)<1$.
If $\inf_{t\ge t_0}\phi(t)>0$ and 
$$\int\limits_{t}^{\infty}\alpha_1(s)\alpha_2(s)w(s)(\|A_{12}(s)\|
+\|A_{21}(s)\|)\,ds<\infty,$$ then \eqref{1} is stable. If, in addition, either 
$\lim\limits_{t\to\infty}\phi(t)=\infty$ or for some $t_0^*$ the next integral is unbounded
\begin{equation}\label{divergence}
\gathered
\int\limits_{t_0^*}^\infty\frac{q(t)dt}{w(t)(1+\beta_1(t)\beta_2(t)\int\limits_{t}^{\infty}\alpha_1(s)\alpha_2(s)(\|A_{12}(s)\|
+\|A_{21}(s)\|)ds)}
\endgathered
\end{equation}
where we have denoted
$$w(t):=\frac{\|g_1\|_{\omega_1,t}}{\omega_1(t)}+
\frac{\|g_2\|_{\omega_2,t}}{\omega_2(t)}+\frac{\|g_1\|_{\omega_1,t}}{\omega_2(t)}+
\frac{\|g_2\|_{\omega_2,t}}{\omega_1(t)},$$
then  \eqref{1} is asymptotically stable.
\end{corollary}
\begin{proof}
Let $t_0^*$ be such that $r_\sigma(\Pi(t_0^*))<1$, then the estimate \eqref{estimate-for-solutions} holds for any $t_0\ge t_0^*$ and due to $\inf_{t\ge t_0}\phi(t)>0$ we
arrive from \eqref{estimate-for-solutions} to $\|x(t;t_0^*,x_0)\|\le C \|x_0\|$ for some suitable constant $C>0$. 

For any other $t_0\in[\tau,t_0^*)$ the same estimate (with a different constant $C$) holds due to $\Omega(t_0,t)=\Omega(t_0,t_0^*)\Omega(t_0^*,t)$ and 
$$\|x(t;t_0,x_0)\|=\|\Omega(t,t_0)x_0\|=\|\Omega(t,t_0^*)\Omega(t_0^*,t_0)x_0\|$$
$$\le \|\Omega(t,t_0^*)\|\cdot\|\Omega(t_0^*,t_0)\|\cdot \|x_0\|\le\tilde{C}\|x_0\| $$
for some $\tilde{C}>0$, which implies the stability.
\\
If $\lim\limits_{t\to\infty}\phi(t)=\infty$, then similarly we obtain for any $t_0\in[\tau,\infty)$ that 
$$\|x(t;t_0,x_0)\|\le \frac{C}{\sqrt{\phi(t)}}\|x_0\| $$
implying the asymptotic stability. The same arguments can be implied in case of the very last condition.
\\
Note that the choice of $\tilde C$ in depends on $t_0$, hence in general the proved asymptotic stability is not necessarily uniform.

It remains to consider the condition \eqref{divergence}. Let us first note, that for some suitable constant $C>0$ we have for all $t\ge0$  $|h_{ii}(t,t_0)|\le C w(t)$ and $h_{12}(t,t_0)$
$$\le C\beta_1(t)\beta_2(t)\int\limits_{t}^{\infty}\alpha_1(s)\alpha_2(s)w(s)(\|A_{12}(s)\|
+\|A_{21}(s)\|)ds$$ Hence \eqref{divergence} implies \eqref{stabC}, guaranteeing the (non-uniform) asymptotic stability.
\end{proof}
Next we provide two examples demonstrating the application of our results and showing advantages over the existing approaches.
\section{Example 1 (from \cite{KaT04})}\label{ex1}
Consider the system of equations for scalar valued functions $x_1,x_2$ with parameters $a_1,a_2\in\mathbb R$, $K>0$, and any initial time $t_0\ge0$, taken from \cite{KaT04}:
\begin{equation}\label{Ex-of-Karafylis}
    \left\{
    \begin{array}{l}
       \dot x_1=-x_1 + a_1 e^{-t}x_2, \\
       \dot x_2=a_2(1+t)x_1-Kx_2.
    \end{array}
    \right.
\end{equation}
It was proved in \cite{KaT04}, this system is non-uniformly asymptotically stable if
\begin{equation}\label{condition-of-Karafyllis}
|a_1a_2|<\sqrt{e^3}K/8.
\end{equation}
We will apply our Corollary \ref{sledstvie} to demonstrate that condition \eqref{condition-of-Karafyllis} can be removed completely.
A comparison system of Wazewski type can be derived for \eqref{Ex-of-Karafylis} using Lyapunov functions $V_i(x_i)=|x_i|,\;i=1,2$, which leads to
\begin{equation}\label{comparison-systems-for-Karafyllis-ex}
    \left\{
    \begin{array}{l}
       \dot x_1=-x_1 + |a_1|e^{-t}x_2, \\
       \dot x_2=|a_2|(1+t)x_1-Kx_2.
    \end{array}
    \right.
\end{equation}
For which we have 
$$\alpha_1(t)=e^{-t},\quad \beta_1(s)=e^{s},\quad \gamma_1(t)=e^{t},\quad \delta_1(s)=e^{-s}$$
$$\alpha_2(t)=e^{-Kt},\quad \beta_2(s)=e^{Ks},\quad \gamma_2(t)=e^{Kt},\quad \delta_2(s)=e^{-Ks}$$
Due to the unbounded interconnection coefficient in the second equation we use the following weighting functions in the definition of integral gains
$$\omega_1(t)=1,\qquad\omega_2(t)=(1+t)^2$$
and after basic calculations we derive from \eqref{gains-ii} and \eqref{gains-ij} that the matrix $\Pi(t_0)$ composed of integral gains can be chosen as
$$\Pi(t_0)=\begin{pmatrix}
  \frac{2|a_1a_2|}{9(K+2)}e^{-t_0}(3t_0+4) & \frac{(1+t_0)^2a_1^2}{(K+2)(K+1)}e^{-2t_0} \\
   \frac{a_2^2}{2(K+1)} & \frac{2|a_1a_2|}{(K+1)(2K+1)}e^{-t_0}
\end{pmatrix}$$
for which we have 
$$r_{\sigma}\Big(\lim\limits_{t_0\to\infty}\Pi(t_0)\Big)=r_{\sigma}
\begin{pmatrix}
0&0\\
\frac{a_2^2}{2(K+1)}&0
\end{pmatrix}=0<1.$$
We chose the weights $q_i(t)=1+t$ and calculate from \eqref{weights_q_i} that 
$$\phi_1(t)=\frac12t+\frac34,\quad \phi_2(t)=\frac{t}{2K}+\frac{2K+1}{4K^2}.$$
Hence $\phi(t):=\min\{\phi_1(t),\phi_2(t)\}$ satisfies $\inf_{t\ge t_0}\phi(t)>0$ and
$\lim\limits_{t\to\infty}\phi(t)=\infty$. Now by Corollary \ref{sledstvie} we conclude that the trivial solution of \eqref{Ex-of-Karafylis} is non-uniformly asymptotically stable.

Note that by our result we obtain the same property as in \cite{KaT04} without \eqref{condition-of-Karafyllis}, which demonstrates an advantage of our results. We also note that the small-gain condition (5.16a) was possible to apply in spite of the unbounded coupling $a_2(1+t)$ due to the exponential decay of the other coupling function $e^{-t}$. In the next example we consider an interconnection with both unbounded coupling functions.
\section{Example 2}\label{ex2}
Consider two scalar interconnected differential equations
\begin{equation}\label{Example2}
\gathered
\dot x_1=-\nu_1 x_1+a_1\psi(t)x_2,\\
\dot x_2=-\nu_2 x_2+a_2\psi(t)x_1
\endgathered
\end{equation}
with $\nu_i>0$, $a_i>0$ and unbounded coupling function
\begin{equation*}
\gathered
\psi(t)=
\begin{cases}
n,\quad t\in[n,n+1/n),\quad n\ge 2,\\
0,\quad t\notin \bigcup\limits_{n=2}^{\infty}[n,n+1/n).
\end{cases}
\endgathered
\end{equation*}
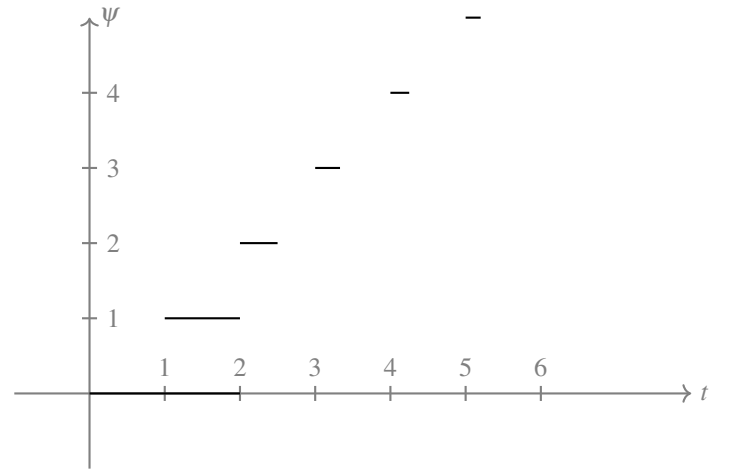
\begin{tikzpicture}
\draw[gray, thick,->] (-1,0) -- (8,0)node[anchor=west]{$t$};
\draw[gray, thick,->] (0,-1) -- (0,5)node[anchor=west]{$\psi$};

\draw[gray, thick,-] (-0.1,1) -- (0.1,1)node[anchor=west]{$1$};
\draw[gray, thick,-] (-0.1,2) -- (0.1,2)node[anchor=west]{$2$};
\draw[gray, thick,-] (-0.1,3) -- (0.1,3)node[anchor=west]{$3$};
\draw[gray, thick,-] (-0.1,4) -- (0.1,4)node[anchor=west]{$4$};

\draw[gray, thick,-] (1,-0.1) -- (1,0.1)node[anchor=south]{$1$};
\draw[gray, thick,-] (2,-0.1) -- (2,0.1)node[anchor=south]{$2$};
\draw[gray, thick,-] (3,-0.1) -- (3,0.1)node[anchor=south]{$3$};
\draw[gray, thick,-] (4,-0.1) -- (4,0.1)node[anchor=south]{$4$};
\draw[gray, thick,-] (5,-0.1) -- (5,0.1)node[anchor=south]{$5$};
\draw[gray, thick,-] (6,-0.1) -- (6,0.1)node[anchor=south]{$6$};

\draw[thick,-] (0,0) -- (2,0);
\draw[thick,-] (1,1) -- (2,1);
\draw[thick,-] (2,2) -- (2.5,2);
\draw[thick,-] (3,3) -- (3.33,3);
\draw[thick,-] (4,4) -- (4.25,4);
\draw[thick,-] (5,5) -- (5.2,5);
\draw[thick,-] (6,6) -- (6.15,6);
\end{tikzpicture}

Since the interconnecting gains are both unbounded and can be of the same sign the results of \cite{LiS10} cannot be applied in order to conclude anything about stability properties of this interconnection. Also the small gain condition (5.16a) from \cite{KaT04} cannot be satisfied. However our approach provides reasonable stability conditions for this interconnection.

The integral gains can be calculated similarly to the previous example and look as follows
\begin{equation*}
\gathered
\pi_{11}(t)=\frac{2a_1a_2e^{(2\nu_1+|\nu_2-\nu_1|)/[t]}}{(1-e^{-(\nu_1+\nu_2)})(1-e^{-2\nu_1})}.
\endgathered
\end{equation*}
\begin{equation*}
\gathered
\pi_{12}(t)=\frac{2a_2^2e^{(2\nu_1+|\nu_2-\nu_1|)/[t]}}{(1-e^{-(\nu_1+\nu_2)})(1-e^{-2\nu_1})},\\
\pi_{21}(t)=\frac{2a_1^2e^{(2\nu_2+|\nu_2-\nu_1|)/[t]}}{(1-e^{-(\nu_1+\nu_2)})(1-e^{-2\nu_2})},\\
\pi_{22}(t)=\frac{2a_1a_2e^{(2\nu_2+|\nu_2-\nu_1|)/[t]}}{(1-e^{-(\nu_1+\nu_2)})(1-e^{-2\nu_2})}
\endgathered
\end{equation*}
For $t\in[p+1/p,p+1)$, $p\ge 2$, it holds that $\pi_{ij}(t)=\pi_{ij}(p+1)$, and for
$t\in[0,2)$ we have $\pi_{ij}(t)=\pi_{ij}(2)$.

This implies that
\begin{equation*}
\gathered
\Pi_1^*(t_0)\\=\frac{1}{1-e^{-(\nu_1+\nu_2)}}
\begin{pmatrix}
\frac{2a_1a_2e^{(2\nu_1+|\nu_2-\nu_1|)/[t_0]}}{1-e^{-2\nu_1}}&\frac{2a_2^2e^{(2\nu_1+|\nu_2-\nu_1|)/[t_0]}}{1-e^{-2\nu_1}}\\
\frac{2a_1^2e^{(2\nu_2+|\nu_2-\nu_1|)/[t_0]}}{1-e^{-2\nu_2}}&\frac{2a_1a_2e^{(2\nu_2+|\nu_2-\nu_1|)/[t_0]}}{1-e^{-2\nu_2}}
\end{pmatrix}
\endgathered
\end{equation*}
having the limit
\begin{equation*}
\gathered
\overline{\Pi}_1=\lim\limits_{t_0\to+\infty}\Pi_1^*(t_0)=\frac{1}{1-e^{-(\nu_1+\nu_2)}}
\begin{pmatrix}
\frac{2a_1a_2}{1-e^{-2\nu_1}}&\frac{2a_2^2}{1-e^{-2\nu_1}}\\
\frac{2a_1^2}{1-e^{-2\nu_2}}&\frac{2a_1a_2}{1-e^{-2\nu_2}}
\end{pmatrix}
\endgathered
\end{equation*}
so that stability condition $\tr\overline{\Pi}_1<1$ takes the form
\begin{equation*}
\gathered
a_1a_2<\frac{(1-e^{-(\nu_1+\nu_2)})(1-e^{-2\nu_1})(1-e^{-2\nu_2})}{2(2-e^{-2\nu_2}-e^{-2\nu_1})}.
\endgathered
\end{equation*}

\section{Conclusions}\label{conclusions}
This work introduces the notion of integral gains for non-autonoumous linear systems. A new small-gain type theorem is proved on the base of this notion. Two examples demonstrate clear advantages of the proposed results. In particular our results can cope with the situations, where known results cannot be applied. A new construction of a coercive Lyapunov function is proposed. The examples are finite dimensional, however our results are valid for infinite dimensional systems. 

In the future we plan to extend these results to the case of unbounded operators to allow the application to interconnected PDEs and to the case of unbounded coupling operators. This will require to deal with non-coercive Lyapunov functions.
Furthermore, we will extend these results to an arbitrary number of inteconnected systems and infinite networks. Also we believe that the idea of Integral gains can be extended to nonlinear systems, which a matter of our future research.

\addtolength{\textheight}{10cm}   









\bibliographystyle{IEEEtran}
\bibliography{libraryECC}




\end{document}